\documentclass[a4paper,10pt]{article}
\usepackage[francais]{babel}
\usepackage[latin1]{inputenc}
\usepackage[T1]{fontenc}
\usepackage[sans]{dsfont}
\usepackage{mathrsfs}
\usepackage{amssymb, amsmath, amsthm}
\usepackage[all]{xy}
\usepackage{multicol}
\usepackage{fullpage}

\makeatletter
\def\section{\@ifstar\unnumberedsection\numberedsection}
\def\numberedsection{\@ifnextchar[%]
  \numberedsectionwithtwoarguments\numberedsectionwithoneargument}
\def\unnumberedsection{\@ifnextchar[%]
  \unnumberedsectionwithtwoarguments\unnumberedsectionwithoneargument}
\def\numberedsectionwithoneargument#1{\numberedsectionwithtwoarguments[#1]{#1}}
\def\unnumberedsectionwithoneargument#1{\unnumberedsectionwithtwoarguments[#1]{#1}}
\def\numberedsectionwithtwoarguments[#1]#2{%
  \ifhmode\par\fi
  \removelastskip
  \vskip 5ex\goodbreak
  \refstepcounter{section}%
  \hbox to \hsize{\hss\vbox{\advance\hsize by 1cm
      \noindent
      \leavevmode\huge\bfseries\raggedright
      \thesection.\ 
      #2\par
      \vskip -2ex
      \noindent\hrulefill
      }}\nobreak
  \vskip 2ex\nobreak
  }
\def\unnumberedsectionwithtwoarguments[#1]#2{%
  \ifhmode\par\fi
  \removelastskip
  \vskip 5ex\goodbreak
  \hbox to \hsize{\hss\vbox{\advance\hsize by 1cm
      \noindent
      \leavevmode\huge\bfseries\raggedright
      #2\par
      \vskip -2ex
      \noindent\hrulefill
      }}\nobreak
  \vskip 2ex\nobreak
  }
\makeatother

\newtheoremstyle{THEOREME}{}{}{\sffamily}{}{\bfseries\scshape}{.}{.5em}{}
\theoremstyle{THEOREME}
\newtheorem{theoreme}{Th\'eor\`eme}[section]
\newtheorem{corollaire}[theoreme]{Corollaire}
\newtheorem{proposition}[theoreme]{Proposition}

\newtheoremstyle{DEFINITION}{0.5cm}{0.5cm}{\itshape}{}{\bfseries}{.}{.5em}{}
\theoremstyle{DEFINITION}
\newtheorem{definition}[theoreme]{D\'efinition}

\newtheoremstyle{REMARQUE}{}{}{\normalfont}{}{\bfseries}{.}{.5em}{}
\theoremstyle{REMARQUE}
\newtheorem{remarque}[theoreme]{Remarque}

\newtheoremstyle{LEMME}{}{}{\slshape}{}{\bfseries}{.}{.5em}{}
\theoremstyle{LEMME}
\newtheorem{lemme}[theoreme]{Lemme}

\input cyracc.def
\font\tencyr=wncyr6
\def\cyr{\tencyr\cyracc}
\newcommand{\cyrb}{{\cyr B}}

\newcommand{\SW}[1]{\textbf{\textit{\emph{\textsf{#1}}}}}%'Special Word', bien mieux que '\emph' pour souligné les définitions
\newcommand{\SWm}[1]{{\textit{\emph{\textsf{#1}}}}}%'Special Word', bien mieux que '\emph' mais pas en gras
\newcommand{\dpl}[1]{\displaystyle{#1}}%Megamode mathematiques

\def\ie{\textit{i.e. }}
\def\cf{\textit{c.f. }}

\newcommand{\Ens}{{\mathfrak{Ens}}}
\newcommand{\A}{\mathscr{A}}
\newcommand{\B}{\mathscr{B}}
\newcommand{\CC}{\mathscr{C}}
\newcommand{\DD}{\mathscr{D}}
\newcommand{\E}{\mathfrak{E}}
\renewcommand{\L}{\mathscr{L}}

\newcommand{\e}{\mathfrak{e}}
\renewcommand{\H}{\mathscr{H}}
\newcommand{\G}{\mathscr{G}}
\newcommand{\opp}{\textsf{opp}}%pour les catégories opposés
\newcommand{\Hom}{\SWm{Hom}}
\newcommand{\HHom}{\SWm{\underline{Hom}}}%Hom interne
\newcommand{\Id}{\SWm{Id}}

\newcommand{\CCp}[1]{\CC_{{w}\geqslant {#1}}}
\newcommand{\CCn}[1]{\CC_{{w}\leqslant {#1}}}
\newcommand{\CCe}[1]{\CC_{{w}= {#1}}}

\newcommand{\DMB}[1]{{\SWm{DM}_{\text{\cyrb}}(#1)}}%motifs de Beilinson
\newcommand{\DMBc}[1]{{\SWm{DM}_{\text{\cyrb},c}(#1)}}%motifs de Beilinson compact
\newcommand{\Lev}{{\ell}}

\newcommand{\DMBcL}[1]{{\SWm{DM}_{\text{\cyrb},c,L}(#1)}}%motifs de Beilinson-Levine constructible
\newcommand{\DMBp}[1]{{{\SWm{DM}_{\text{\cyrb}}(#1)}_{W\geqslant 0}}}%motifs de Beilinson de poids positif
\newcommand{\DMBn}[1]{{{\SWm{DM}_{\text{\cyrb}}(#1)}_{W\leqslant 0}}}%motifs de Beilinson de poids négatif
\newcommand{\DMBe}[1]{{{\SWm{DM}_{\text{\cyrb}}(#1)}_{W= 0}}}%le coeur des motifs de Beilinson
\newcommand{\DMBcp}[1]{{{\SWm{DM}_{\text{\cyrb},c}(#1)}_{w\geqslant 0}}}%motifs de Beilinson compact de poids positif
\newcommand{\DMBcn}[1]{{{\SWm{DM}_{\text{\cyrb},c}(#1)}_{w\leqslant 0}}}%motifs de Beilinson compact de poids négatif
\newcommand{\DMBce}[1]{{{\SWm{DM}_{\text{\cyrb},c}(#1)}_{w= 0}}}%le coeur des motifs de Beilinson compact

\newcommand{\Neg}{\SWm{NEG}}%motifs négatifs de Beilinson
\newcommand{\Pos}{\SWm{POS}}%motifs positifs de Beilinson

\newcommand{\N}{\mathbb{N}}
\newcommand{\Z}{\mathbb{Z}}
\newcommand{\Q}{\mathbb{Q}}

\newcommand{\morph}[2]{#1 \rightarrow #2}%Une fleche entre deux objets
\newcommand{\morphp}[3]{#1 \rightarrow #2 \rightarrow #3}%Deux fleches entre trois objets
\newcommand{\immouv}[1][r]{\ar@{^(->}[#1]|*[@]{\circ}\ar[#1]}
\newcommand{\immfer}[1][r]{\ar@{^(->}[#1]|*[@]{\shortmid}\ar[#1]}
\newcommand{\immop}[2]{\raisebox{.7ex}{\xymatrix{#1 \immouv[r]& #2}}}%Une immersion ouverte
\newcommand{\immcl}[2]{\raisebox{.7ex}{\xymatrix{#1 \immfer[r]& #2}}}%Une immersion ferme
\newcommand{\tridis}[3]{#1 \rightarrow #2 \rightarrow #3 \overset{+1}{\longrightarrow}}%triangle distingué
\def\cartesien{\ar@{}[rd]|{\square}}

\newcommand{\Kar}{\mathfrak{R}}
\newcommand{\SE}{{\mathfrak{Ext}}}
\newcommand{\CA}[1]{{\left\langle{#1}\right\rangle}}%Catégorie engendré
\newcommand{\CAep}[1]{{\left\langle{#1}\right\rangle}^{\text{ép}}}%Catégorie épaise engendré
\newcommand{\CAinf}[1]{{\left\langle{#1}\right\rangle}_{\infty}}%Catégorie infini engendré
\newcommand{\CAepinf}[1]{{\left\langle{#1}\right\rangle}^{\text{ép}}_{\infty}}%Catégorie épaise engendré

\newcommand{\Un}{\mathds{1}}%unité du produit tensoriel
\newcommand{\Gr}{\SWm{Gr}}%Gradué
\newcommand{\KK}{\SWm{K}}%K de K-theorie
\newcommand{\Spec}{\SWm{Spec}}%Spectre
\newcommand{\AAA}{\mathbb{A}}%droite affine
\newcommand{\restr}[2]{{#1}_{{|}_{#2}}}%restriction
\title{Structures de poids à la Bondarko\\
sur les motifs de Beilinson}
\author{David H\'ebert}
\date{}

\begin{document}

\maketitle

\renewcommand{\abstractname}{\textsc{Abstract}}
\begin{abstract}
\noindent Bondarko defines and studies the notion of weight structure and he shows that there exists a weight structure over the category of Voevodsky motives with rational coefficients (over a field of characteristic $0$). In this paper we extend this weight structure to the category of Beilinson motives (for any scheme of finite type over a base scheme which is excellent of dimension at most two) introduced and studied by Cisinsky-D\'eglise. We also check the weight exactness of the Grothendieck operations.
\end{abstract}

\tableofcontents

\section*{Introduction.}
\addcontentsline{toc}{section}{Introduction.}

Dans \cite{Bo}, Bondarko définit et étudie la notion de \SW{structure de poids}. Il montre qu'il existe une structure de poids sur la catégorie des motifs à la Voevodsky à coefficients rationnels définie sur un corps de caractéristique $0$ (\cf \cite{VSF}) dont le coeur s'identifie à la catégorie des motifs de Chow sur $k$. La question qui se pose alors (\cf \cite[rm. 8.2.5.3]{Bo}) est de savoir comment généraliser cette structure de poids à la catégorie des motifs de Beilinson, introduite et étudiée par Cisinski-Déglise (\cf \cite{CD}).

Dans la première partie, nous redonnons la définition de structure de poids (définition \ref{def.WS}) ainsi que la preuve du théorème de construction de Bondarko (théorème \ref{thm.constr}). La seconde partie est entièrement dédiée au rappel du formalisme des six opérations de Grothendieck dans la catégorie des motifs de Beilinson. 
L'apport nouveau de cet article réside dans la troisième partie dans laquelle nous construisons une structure de poids sur les motifs de Beilinson (théorème \ref{cor.constr.WS.motifs}) répondant ainsi positivement à la question posée par Bondarko. Pour finir, nous établissons les propriétés de stabilité par les six opérations (théorème \ref{propo.compatibilite}).

Cet article précéde \cite{Bo2} dans lequel Bondarko donne une autre preuve du théorème \ref{cor.constr.WS.motifs}.

\noindent \textbf{Notations et conventions.}

Si $\CC$ est une catégorie, la notation $\H\subset \CC$ (où $\CC\supset\H$) signifiera toujours que $\H$ est une sous-catégorie pleine de $\CC$. Pour cette raison nous décrirons les sous-catégories pleines uniquement par la classe de leurs objets. Nous adopterons également les notations ensemblistes ($\in$, $\exists$, $\cup$, $\cap$, etc.) pour les catégories. Par exemple, la notation $X\in \CC$ signifiera toujours que $X$ est un objet de $\CC$. Les triangles distingués seront notés $\tridis{A}{B}{C}$.

On note $\H_c$ la sous-catégorie pleine de $\H$ formée des objets compacts de $\H$ ; on rappelle qu'un objet $H\in \H$ est compact si $\Hom_\H(H,\bullet)$ commute aux petites sommes.

Tous les schémas considérés sont de type fini sur une base $B$ excellente de dimension de Krull au plus $2$. Les morphismes entre schémas sont séparés.

\newpage

\section{Structure de poids.}
\addcontentsline{toc}{section}{\protect\numberline{\thesection} Structure de poids.}
On fixe $\CC$ une catégorie triangulée (on note $[1]$ son foncteur de translation) et $\A$, $\B$ et $\H$ des sous-catégories pleines de $\CC$ possédant $0$ (l'objet initial et final de $\CC$).

\begin{definition} On considère les sous-catégories pleines de $\CC$ suivantes :
\begin{description}
\item[$(i)$. ]\SW{L'enveloppe des rétractes} de $\H$, notée $\Kar(\H)$, est
$$\Kar(\H) :=\left\{X\in \CC\big| \exists \left(\morphp{X}{H}{X}=\Id_X\right),\; H\in \H\right\}.$$
\item[$(ii)$. ] \SW{L'orthogonal à droite} (resp. \SW{à gauche}) de $\H$, notée $\H^\bot$ (resp. ${{}^\bot}\H$), est
\begin{eqnarray*}
&&\H^\bot :=\left\{X\in \CC\big| \forall H\in \H,\;\Hom_\CC(X,H)=0\right\}.\\
\big(resp.&& {{}^\bot}\H :=\left\{X\in \CC\big| \forall H\in \H,\;\Hom_\CC(H,X)=0\right\}.\big)
\end{eqnarray*}
\item[$(iii)$. ] La catégorie des \SW{$1$-extensions} de $\B$ par $\A$, notée $\SE_\CC^1(\B,\A)$, est
$$\SE_\CC^1(\B,\A) :=\left\{X\in \CC\big| \exists \left(\tridis{A}{X}{B}\right),\; A\in \A,\,B\in\B\right\}.$$
On pose $\SE_\CC^1(\H)=\SE^1_\CC(\H,\H)$.
\item[$(iv)$. ] On pose $\dpl{\SE_\CC(\H)=\bigcup_{n\in \N}\SE_\CC^n(\H)}$, appelée \SW{enveloppe des extensions} de $\H$ où
\begin{eqnarray*}
\SE_\CC^0(\H)&=&\H,\\
\forall n\in \N,\quad\SE_\CC^{n+1}(\H)&=&\SE_\CC^1(\SE_\CC^n(\H)).
\end{eqnarray*}
\item[$(v)$. ] On note $\CA{\H}$ la catégorie \SW{engendrée} par $\H$,
$$\CA{\H}:=\SE_\CC\left(\bigcup_{n\in \Z}\H[n]\right).$$
\item[$(vi)$. ] On note $\CAep{\H}$ la catégorie \SW{épaisse engendrée} par $\H$,
$$\CAep{\H}:=\Kar\left(\CA{\H}\right).$$
\item[$(vii)$.] On note $\H^\oplus$, \SW{l'enveloppe additive} de $\H$,
$$\H^\oplus:=\left\{\bigoplus_{i=0}^nH_i\big|n\in \N,\ \forall i\in [\![0,n]\!],\ H_i\in \H\right\}\cup\{0\}.$$
\end{description}
\end{definition}

\begin{remarque}
Les objets de $\Kar(\H)$ sont en fait les facteurs directs d'objets de $\H$.

La catégorie $\CA{\H}$ est la plus petite sous-catégorie triangulée de $\CC$ contenant $\H$.

La catégorie $\CAep{\H}$ est la plus petite sous-catégorie épaisse et triangulée de $\CC$ contenant $\H$.

La catégorie $\H^\oplus$ est la plus petite sous-catégorie additive de $\CC$ contenant $\H$.
\end{remarque}

\begin{definition}$ $
\begin{description}
\item[$(i)$. ] On dira que $\H$ est \SW{stable par rétractes} si $\H=\Kar(\H)$.
\item[$(ii)$. ] On dira que $\H$ est \SW{stable par extensions} si $\H=\SE_\CC^1(\H)$.
\item[$(iii)$.] On dira que $(\A,\B)$ est une \SW{pondération} de $\H$ si $\H\subset\SE^1_\CC(\B,\A)$.
%\item[$(iv)$. ] On dira que $\H$ \SW{engendre} $\CC$ si $\CAep{\H}=\CC$.
\end{description}
\end{definition}

\begin{remarque}\label{rem...}
Tout orthogonal (à gauche ou a droite) est stable par rétracte. 

La catégorie $\SE_\CC^1(\B,\A)$ est la plus petite catégorie telle que $(\A,\B)$ en soit une pondération.
\end{remarque}

\begin{definition}[{comp. \cite[déf. 1.1.1]{Bo}}]\label{def.WS}
On dira que $w=(\CCn{0},\CCp{0})$, où $\CCn{0}$, $\CCp{0}\subset \CC$, est une \SW{structure de poids} sur $\CC$, notée $w/\CC$, si les axiomes suivants sont satisfaits :
\begin{description}
\item[$(SP1)$. Stabilité par rétractes.] $$\Kar(\CCn{0})=\CCn{0},\quad\Kar(\CCp{0})=\CCp{0}.$$
\item[$(SP2)$. Semi-invariance avec respect des translations.] $$\CCn{0}[-1]\subset\CCn{0},\quad\CCp{0}[1]\subset\CCp{0}.$$
\item[$(SP3)$. Orthogonalité faible.] $$\CCn{0}\subset\CCp{0}[1]^\bot.$$
\item[$(SP4)$. Filtration par le poids.] La donnée $(\CCn{0},\CCp{0}[1])$ est une pondération de $\CC$. On appellera un triangle $\tridis{A}{X}{B}$ où $X\in \CC$, $A\in \CCn{0}$ et $B\in\CCp{0}[1]$, une \SW{filtration par le poids} de $X$.
\end{description}
Pour tout $n\in \Z$, on note
$$\CCn{n}:=\CCn{0}[n],\quad \CCp{n}:=\CCp{0}[n],\quad\CCe{n}:=\CCn{n}\cap\CCp{n}.$$
On appelle $\CCe{0}$ le \SW{c{\oe}ur} de la structure de poids. 
\end{definition}

A noter que la notion de structure de poids fut indépendamment introduite par Pauksztello dans \cite{Pau} alors appélée \SW{co-$t$-structure}.

\begin{proposition}[Orthogonalité forte ; comp. {\cite[prop. 1.3.3.1]{Bo}}] Soit $w/\CC$ une structure de poids.
$$\CCn{0}=\CCp{1}^\bot,\quad\CCp{0}={{}^\bot}\CCn{-1}.$$
\end{proposition}

\begin{proof}
Soient $X\in \CCp{1}^\bot$ et $\tridis{A}{X}{B}$ une filtration par le poids de $X$. Via le foncteur cohomologique $\Hom_\CC(X,\bullet)$  on obtient la suite exacte $\morphp{\Hom_\CC(X,A)}{\Hom_\CC(X,X)}{\Hom_\CC(X,B)}$. Or $\Hom_\CC(X,B)=0$ d'où un épimorphisme ${\Hom_\CC(X,A)}\twoheadrightarrow{\Hom_\CC(X,X)}$ qui permet de voir $X$ comme un rétracte de $A\in \CCn{0}$ et $X\in \CCn{0}$ par $(SP1)$.
\end{proof}

\`A présent nous allons établir un théorème de construction de structure de poids.

\begin{lemme}\label{orth.ext.stab}
On a les égalités suivantes
$$\SE_\CC(\H^\bot)=\SE_\CC(\H)^\bot=\H^\bot, \qquad\SE_\CC({^\bot}\H)={^\bot}\SE_\CC(\H)={^\bot}\H.$$
\end{lemme}

\begin{proof}
On a $\H\subset \SE_\CC(\H)$ et l'opération d'orthogonalité inversant les inclusions on aboutit trivialement à $\SE_\CC(\H)^\bot\subset\H^\bot\subset\SE_\CC(\H^\bot)$. Pour montrer que ces inclusions sont des égalités, on va montrer par r\'ecurrence sur $n\in \N$ l'\'enonc\'e suivant :
pour tout entier $m\in \N$,
$\SE^{n}_\CC(\H^{\bot})\subset\SE^{m}_\CC(\H)^{\bot}.$
\begin{description}
\item[Cas initial : $n=0$.] R\'ecurrence sur $m$ ;
le cas $m=0$ \'etant trivial. Supposons que pour un $m$ quelconque
fix\'e on ait $\H^\bot\subset\SE^{m}_\CC(\H)^\bot$. Soit $X\in\H^\bot$, 
on veut voir qu'il s'agit d'un objet de
$\SE^{m+1}_\CC(\H)^\bot$, c'est \`a dire que pour tout  $Y\in\SE^{m+1}_\CC(\H)$ on ait $\Hom_\CC(X,Y)=0$.
Par d\'efinition on a un triangle distingu\'e de $\CC$
de la forme $\tridis{A}{Y}{B}$ tel que $A, B\in \SE^{m}_\CC(\H)$. Le foncteur $\Hom_\CC(X,\bullet)$ \'etant
cohomologique on en d\'eduit la suite exacte
$\morphp{\Hom_\CC(X,A)}{\Hom_\CC(X,Y)}{\Hom_\CC(X,B)}$. Mais les
objets extr\'emaux de cette suite sont nuls car $X\in\H^\bot$ et $A, B\in\SE^{m}_\CC(\H)$
dont, par hypoth\`ese de r\'ecurrence, nous savons que
$\H^\bot\subset\SE^{m}_\CC(\H)^\bot$. Ainsi $\Hom_\CC(X,Y)=0$.
\item[R\'ecurrence.]
On va montrer que quelque soit l'entier $m\in \N$ on a
$\SE^{n+1}_\CC(\H^\bot)\subset\SE^{m}_\CC(\H)^\bot$. Soit $X\in\SE^{n+1}_\CC(\H^\bot)$. On veut voir qu'il est dans
$\SE^{m}_\CC(\H)^\bot$, c'est \`a dire que pour tout $Y\in\SE^{m}_\CC(\H)$, on ait $\Hom_\CC(X,Y)=0$. Il existe par
d\'efinition $\tridis{A}{X}{B}$ tel que $A, B\in \SE^{n}_\CC(\H^\bot)$. Le foncteur $\Hom_\CC(\bullet,Y)$ \'etant
cohomologique, on en d\'eduit une suite exacte
$\morphp{\Hom_\CC(B,Y)}{\Hom_\CC(X,Y)}{\Hom_\CC(A,Y)}$. Comme $A, B\in \SE^{n}_\CC(\H^\bot)$ qui, par hypoth\`ese
de r\'ecurrence, est inclus dans $\SE^{m}_\CC(\H)^\bot$, on en
d\'eduit que les deux objets extr\'emaux de cette suite sont nuls et
donc que $\Hom_\CC(X,Y)=0$.
\end{description}
On raisonne dualement pour l'orthogonal à gauche.
\end{proof}

\newpage

\begin{proposition}\label{lem.constr}
Supposons $\A\subset\B[1]^\bot$. Si $(\A,\B)$ est une pondération de $\H$ alors $(\SE_\CC(\A),\SE_\CC(\B))$ est une pondération de $\SE_\CC(\H)$.
\end{proposition}

\begin{proof}
Comme $\dpl{\SE_\CC(\H)=\bigcup_{n\in \N}\SE_\CC^n(\H)}$, on va raisonner par récurrence sur $n\in \N$ ; le cas initial $n=0$ suit de l'hypothèse de l'énoncé. 
Soit $X\in \SE_\CC^{n+1}(\H)$ ; par construction il existe un triangle distingué $\tridis{X'}{X}{X''}$ avec $X'$ et $X''$ des objets de $\SE_\CC^n(\H)$ dont, par hypothèse de récurrence, $(\SE_\CC(\A),\SE_\CC(\B))$ est une pondération. C'est à dire qu'il existe  $A',A''\in\SE_\CC(\A)$ et $B',B''\in\SE_\CC(\B)$\\
\vspace{-0.7cm}
\setlength{\columnseprule}{1pt}
\begin{multicols}{2}
\noindent telle que l'on ait le diagramme suivant
$$\xymatrix@R=0.5cm{
A'\ar[d]&&A''\ar[d]&A'[1]\ar[d]\\
X'\ar[d]\ar[r]&X\ar[r]&X''\ar[d]\ar[r]&X'[1]\ar[d]\\
B'\ar[d]^{+1}&&B''\ar[d]^{+1}&B'[1]\ar[d]^{+1}\\
&&&
}$$
Comme $\SE_\CC(\A)\subset\SE_\CC(\B[1]^\bot)\overset{\ref{orth.ext.stab}}{=}\SE_\CC(\B)[1]^\bot$ on peut appliquer \cite[prop. 1.1.9]{BBD} (sur la partie droite du diagramme) et \cite[prop. 1.1.11]{BBD}, pour compléter le précédent diagramme en
$$\xymatrix@R=0.5cm{
A'\ar[d]\ar[r]&A\ar[r]\ar[d]&A''\ar[d]\ar[r]^{+1}&\\
X'\ar[d]\ar[r]&X\ar[r]\ar[d]&X''\ar[d]\ar[r]^{+1}&\\
B'\ar[r]\ar[d]_{+1}&B\ar[r]\ar[d]_{+1}&B''\ar[r]_{+1}\ar[d]_{+1}&\\
&&&
}$$
où toutes les lignes et toutes les colonnes sont des triangles distingués. La stabilité par extension permet de conclure que $A\in \SE_\CC(\A)$ et $B\in \SE_\CC(\B)$.
\end{multicols}
\end{proof}

\begin{theoreme}[Théorème de construction de Bondarko ; comp. {\cite[thm. 4.3.2.II.1, prop. 5.2.2]{Bo}}]\label{thm.constr}
Supposons que l'une des conditions suivantes soit satisfaite
\begin{center}
$(a)$. $\CC=\CA{\H}$,\hspace{2cm}$(b)$. $\CC$ est pseudo-abélienne et $\CC=\CAep{\H}$.
\end{center}
Alors les conditions suivantes sont équivalentes :

\begin{description}
\item[$(i)$.] Il existe une unique structure de poids $w/\CC$ telle que $\H\subset\CCe{0}$,
\item[$(ii)$.] $\dpl{\H\subset\left(\bigcup_{n>0}\H[n]\right)^\bot}$.
\end{description}
De plus, dans le cas $(b)$, $\CCe{0}=\Kar(\H^\oplus)$.
\end{theoreme}

\begin{proof} 
L'orhtogonalité faible prouve que la condition $(ii)$ soit nécessaire. 

Supposons la condition $(a)$ satisfaite.
Sous $(ii)$ on construit la structure de poids suivante :
$$\CCn{0}=\Kar\left(\SE_\CC\left(\bigcup_{n\leqslant 0}\H[n]\right)\right),\qquad \CCp{0}=\Kar\left(\SE_\CC\left(\bigcup_{n\geqslant 0}\H[n]\right)\right).$$
Les axiomes $(SP1)$ et $(SP2)$ viennent de la construction, $(SP3)$ vient de l'hypothèse $(ii)$, quand à $(SP4)$ 
on considère la pondération triviale sur $\dpl{\overline{\H}:=\bigcup_{n\in \Z}\H[n]}$ : on prend un objet $X$ dans cette catégorie, c'est à dire qu'il est dans l'un des $\H[n]$ ; si $n\leqslant 0$ on considère le triangle $\tridis{X}{X}{0}$, sinon (\ie $n>0$) on considère le triangle $\tridis{0}{X}{X}$. Ainsi, en posant $\dpl{\A = \bigcup_{n\leqslant 0}\H[n]}$ et $\dpl{\B = \bigcup_{n>0}\H[n]}$, $(\A,\B)$ est une pondération de $\overline{\H}$. Grace à $(ii)$, on peut appliquer  \ref{lem.constr} : $(\SE_\CC(\A),\SE_\CC(\B))$ et \textit{a fortiori} $(\Kar(\SE_\CC(\A)),\Kar(\SE_\CC(\B)))$ est une pondération de $\SE_\CC(\overline{\H})=\CA{\H}=\CC$. L'unicité de cette structure suit de l'orthogonalité forte. 

Supposons à présent la condition $(b)$ satisfaite. Quitte à remplacer $\H$ par $\H^\oplus$, on peut supposer que $\H$ est additive.
Notons $\e(\H)$ la petite enveloppe de $\H$ (\cite[déf. 4.3.1.3]{Bo}) et $\E(\H)$ son enveloppe pseudo-abélienne (voir par exemple \cite[déf. 1.2]{Bal.Sch} ; à noter que l'on ne peut prendre ni la petite enveloppe ni l'enveloppe pseudo-abélienne de $\H$ si elle n'est pas additive ; à noter de plus qu'il existe une équivalence de catégorie entre $\Kar(\H)$ et $\E(\H)$) de sorte que l'on ait les inclusions suivantes $\H\subset\e(\H)\subset\E(\H)$ qui sont des égalités lorsque $\H$ est pseudo-abélienne.

Le raisonnement précédent amène une structure de poids $d$ sur $\DD=\CA{\H}$. Puisque c'est le cas de $\DD_{d=0}$ (orthogonalité faible), $\E(\DD_{d=0})$ vérifie la condition $(ii)$, ainsi en appliquant encore le raisonnement précédent il existe une unique structure de poids $d'$ sur $\DD'=\CA{\E(\DD_{d=0})}\subset \E(\DD)$ telle que $\E(\DD_{d=0})\subset\DD'_{d'=0}$. D'aprés \cite[thm. 4.3.2.II.2]{Bo} on a $\E(\DD_{d=0})=\e(\E(\DD_{d=0}))=\DD'_{d'=0}$. Le coeur de $d'$ est pseudo-abélien il en va donc de même pour $\DD'$ (\cf \cite[lm. 5.2.1]{Bo}) et nécessairement $\DD'=\E(\DD)=\CC$.

Nous avons ainsi trouvé une structure de poids sur $w/\CC$ qui est $d'$. En particulier $\CCe{0}=\DD'_{d'=0}=\E(\DD_{d=0})=\E(\e(\H))=\Kar(\H)$.
\end{proof}

\begin{remarque}\label{rm.thm.constr}
Dans le cas de la condition $(b)$ on peut donner explicitement la structure de poids comme dans la condition $(a)$. En reprenant les notations de la preuve précédente, on arrive à 
$$\CCn{0}=\DD'_{d'\leqslant0}=\Kar\left(\SE_{\E(\DD)}\left(\bigcup_{n\leqslant 0}\E(\DD_{d=n})\right)\right).$$
Sachant qu'il existe une équivalence de catégorie entre l'enveloppe des rétractes et l'enveloppe pseudo-abélienne, que $\E(\DD)=\CC$ et que $\DD_{d=0}=\e(\H^\oplus)$ on en déduit
$$\CCn{0}=\Kar\left(\SE_\CC\left(\bigcup_{n\leqslant0}\Kar(\H^\oplus)[n]\right)\right)=\Kar\left(\SE_\CC\left(\Kar\left(\bigcup_{n\leqslant0}\H^\oplus[n]\right)\right)\right).$$
De même en changeant le symbole $\leqslant$ par $\geqslant$.
\end{remarque}

On peux ``ajouter'' aux précédents résultats des petites sommes. On suppose que les objets de $\CC$ sont stables par petites sommes.

\begin{definition}
On considère les sous-catégories pleines de $\CC$ suivantes :
\setlength{\columnseprule}{1pt}
\begin{multicols}{2}
\begin{description}
\item[$(i)$.] 
$\dpl{\H^\infty :=\left\{\bigoplus_{i\in I}H_i\big| I\in \Ens,\;\forall i\in I, H_i\in \H\right\}.}$
\item[$(ii)$.]
$\dpl{\SE^\infty_\CC(\H):=\SE_\CC\left(\SE_\CC(\H)^\infty\right).}$
\item[$(iii)$.] 
$\dpl{\CAinf{\H}:=\SE_\CC^\infty\left(\bigcup_{n\in \Z}\H[n]\right).}$
\item[$(iv)$.] 
$\dpl{\CAepinf{\H}:=\Kar\left(\CAinf{\H}\right).}$
\end{description}
\end{multicols}
\end{definition}

\begin{lemme}\label{lm.clef.orth.inf}
On a les égalités suivantes 
$${^\bot}\H={^\bot}(\H^\infty)\subset({^\bot}\H)^\infty.$$
Si de plus les objets de $\H$ sont compacts (\ie $\H=\H_c$) alors l'inclusion est une égalité.
\end{lemme}

\begin{proof}
Naturellement ${^\bot}\H\subset({^\bot}\H)^\infty$. De même $\H\subset\H^\infty$ ce qui donne
${^\bot}(\H^\infty)\subset{^\bot}\H$. Vérifions l'inclusion inverse : soient $X\in {^\bot}\H$ et $H\in \H^\infty$ c'est à dire qu'il existe un ensemble d'indice $I$ et $H_i\in\H$ indexé par $I$ tel que $\dpl{H=\bigoplus_{i\in I}H_i}$ ; mais $\dpl{\Hom_\CC(H,X)=\prod_{i\in I}\Hom_\CC(H_i,X)=0}$. 

Supposons à présent que les objets de $\H$ sont compacts et montrons que $({^\bot}\H)^\infty\subset{^\bot}(\H^\infty)$ : soient $X\in ({^\bot}\H)^\infty$ et $H\in \H^\infty$ ; cela signifie qu'il existe des ensembles d'indices $I$ et $J$ tel que $H=\dpl{\bigoplus_{i\in I}H_i}$ et $X=\dpl{\bigoplus_{j\in J}X_j}$ où chaque $H_i\in \H$ et $X_j\in {^\bot}\H$, ainsi
\vspace*{-0.7cm}
\begin{eqnarray*}
\Hom_\CC(H,X)&=&\Hom_\CC\left(\bigoplus_{i\in I}H_i,\bigoplus_{j\in J}X_j\right)\\
%&=&\prod_{i\in I}\Hom_\CC\left(H_i,\bigoplus_{j\in J}X_j\right)\\
&\overset{\text{compact}}{=}&\prod_{i\in I}\bigoplus_{j\in J}\Hom_\CC\left(H_i,X_j\right)\\
&=&0.
\end{eqnarray*}
\end{proof}

\begin{lemme}\label{orth.ext.stab.inf}
On a les égalités suivantes
$${^\bot}\H={^\bot}\SE_\CC^\infty(\H)={^\bot}(\SE_\CC(\H)^\infty)\subset(\SE_\CC({^\bot}\H))^\infty.$$
Si de plus $\H=\H_c$ alors l'inclusion est une égalité.
\end{lemme}

\begin{proof}
C'est le lemme précédent et \ref{orth.ext.stab}.
\end{proof}

\begin{lemme}\label{lm.technique}
Supposons que $\A=\A_c$. Alors on a les équivalences suivantes.
$$\xymatrix@R=0.5cm{(\A\subset\B^\bot)\ar@{<=>}[d]\ar@{<=>}[r]&(\A^\infty\subset (\B^\infty)^\bot)\ar@{<=>}[d]\\
(\B\subset{^\bot}\A)\ar@{<=>}[r]&(\B^\infty\subset {^\bot}(\A^\infty))}$$
\end{lemme}

\begin{proof}
Les équivalences verticales sont triviales. Il suffit de vérifier $(\B\subset{^\bot}\A)\Longleftrightarrow(\B^\infty\subset {^\bot}(\A^\infty))$. L'orthogonalité inversant le sens des inclusions on a $\B\subset\B^\infty\subset{^\bot}(\A^\infty)\subset{^\bot}\A$ (ce qui prouve $\Leftarrow$). Pour la réciproque on remarque que ${^\bot}\A={^\bot}(\A^\infty)$ est stable par somme quelconque (\cf \ref{lm.clef.orth.inf}) ; donc si $\B\subset{^\bot}\A$ alors $\B^\infty\subset{^\bot}(\A^\infty)$.
\end{proof}

\begin{proposition}\label{lem.constrinf}
Supposons $\A\subset \B[1]^\bot$ et $\A=\A_c$. Si $(\A,\B)$ est une pondération de $\H$ alors $(\SE_\CC^\infty(\A),\SE_\CC^\infty(\B))$ est une pondération de $\SE_\CC^\infty(\H)$.
\end{proposition}

\begin{proof}
Soit $X\in \SE_\CC(\H)^\infty$, c'est à dire $\dpl{X=\bigoplus_{i\in I}X_i}$ pour un certain ensemble d'indice $I$ ou chaque $X_i\in \SE_\CC(\H)$. D'après $\ref{lem.constr}$, il existe $A_i\in \SE_\CC(A)$, $B_i\in \SE_\CC(\B)$ et un triangle distingué $\tridis{A_i}{X_i}{B_i}$. En sommant ces triangles on obtient le triangle $\tridis{A}{X}{B}$ où $A\in \SE_\CC(\A)^\infty$ et $B\in \SE_\CC(\B)^\infty$. Nous avons ainsi prouvé que $(\SE_\CC(\A)^\infty,\SE_\CC(\B)^\infty)$ est une pondération de $\SE_\CC(\H)^\infty$. D'après le lemme \ref{lm.technique}, comme les objets de $\SE_\CC(\A)$ sont compacts (car extensions de compacts), on a $\SE_\CC(\A)^\infty\subset (\SE_\CC(\B)^\infty[1])^\bot$. On conclut en appliquant encore \ref{lem.constr}.
\end{proof}

\begin{theoreme}\label{thm.constr2}
Supposons que $\H=\H_c$ et que l'une des conditions suivantes soit satisfaite
\begin{center}
$(a_\infty)$. $\CC=\CAinf{\H}$,\hspace{2cm}$(b_\infty)$. $\CC=\CAepinf{\H}$.
\end{center}
Alors les conditions suivantes sont équivalentes :
\begin{description}
\item[$(i)$.] Il existe une unique structure de poids $w/\CC$ telle que $\H^\infty\subset\CCe{0}$,
\item[$(ii)$.] $\dpl{\H\subset\left(\bigcup_{n>0}\H[n]\right)^\bot}$.
\end{description}
\end{theoreme}

\begin{proof}
On raisonne comme pour \ref{thm.constr} en ``ajoutant'' des sommes infinies. Dans le cas $(a_\infty)$ on construit la structure de poids suivante :
$$\CCn{0}=\Kar\left(\SE^\infty_\CC\left(\bigcup_{n\leqslant 0}\H[n]\right)\right),\qquad \CCp{0}=\Kar\left(\SE_\CC^\infty\left(\bigcup_{n\geqslant 0}\H[n]\right)\right).$$ On raisonne comme dans le cas $(a)$ de \ref{thm.constr} : soient $\A$, $\B$ et $\overline{\H}$ comme dans la preuve du cas $(a)$. Alors $\A\subset\B[1]^\bot$ par $(ii)$ donc $\SE_\CC(\A)\subset\SE_\CC(\B)[1]^\bot$ ; par \ref{lm.technique} on en déduit $\SE_\CC(\A)^\infty\subset(\SE_\CC(\B)^\infty[1])^\bot$ (les objets de $\SE_\CC(\A)$ sont compacts car c'est le cas des objets de $\A$) ce qui prouve l'axiome $(SP3)$ (via \ref{orth.ext.stab}). Nous avons vu que $(\A,\B)$ est une pondération de $\overline{\H}$ ; on prouve $(SP4)$ en appliquant \ref{lem.constrinf}.

Le cas $(b_\infty)$ se traite comme le cas $(b)$ de \ref{thm.constr} sachant que $\CC$ est pseudo-abélienne (voir par exemple  \cite[prop. 1.6.8]{Nee}).
\end{proof}

\begin{remarque}\label{rm.thm.constr2}
Dans le cas $(b_\infty)$ on peut décrire la structure de poids :
$$\CCn{0}=\Kar\left(\SE_\CC^\infty\left(\Kar\left(\bigcup_{n\leqslant0}\H^\oplus[n]\right)\right)\right),\qquad\CCp{0}=\Kar\left(\SE^\infty_\CC\left(\Kar\left(\bigcup_{n\geqslant0}\H^\oplus[n]\right)\right)\right).$$
\end{remarque}

\begin{definition}[{comp. \cite[déf. 1.2.1.VI]{Bo2}}]
Soient $\CC$ et $\CC'$ des catégories triangulées, $c/\CC$, $c'/\CC'$ des structures de poids et $F:\morph{\CC}{\CC'}$ un foncteur de catégories triangulées.
\begin{description}
\item[$\bullet$] On dira que $F$ est \SW{$w$-exacte à gauche} si $F$ transforme les objets de $\CC_{c \leqslant0}$ en objets de $\CC'_{c' \leqslant0}$.
\item[$\bullet$] On dira que $F$ est \SW{$w$-exacte à droite} si $F$ transforme les objets de $\CC_{c \geqslant0}$ en objets de $\CC'_{c' \geqslant0}$.
\item[$\bullet$] On dira que $F$ est \SW{$w$-exacte} s'il est $w$-exacte à gauche et à droite.
%\item[$\bullet$] On dira que $F$ est \SW{faiblement $w$-exacte} si $F$ transforme les objets de $\CCe{0}$ en objet de $\CC'_{w'=0}$.
\item[$\bullet$] Supposons $\CC'\subset\CC$ ; on dira que $c'$ est une \SW{restriction} de $c$, notée $c'=\restr{c}{\CC'}$ si le foncteur d'inclusion canonique de $\CC'$ dans $\CC$ est $w$-exacte.
\end{description}
\end{definition}

%\begin{proposition}[{comp. \cite[prop. 1.2.3.9]{Bo2}}]\label{prop.ws.adj}
%Soient $\CC$ et $\CC'$ des catégories triangulées, $c/\CC$, $c'/\CC'$ des structures de poids et $G:\morph{\CC}{\CC'}$, $D:\morph{\CC'}{\CC}$ des foncteurs de catégories triangulées tels que $G$ soit l'adjoint à gauche de $D$. Alors $G$ est $w$-exacte à gauche si et seulement si $D$ est $w$-exacte à droite. 
%\end{proposition}

%\begin{proof}
%Supposons par exemple que $D$ soit $w$-exacte à droite ; pour vérifier que $G$ est $w$-exacte à gauche il faut voir que pour tout $N\in \CC_{c \leqslant0}$, $G(N)\in \CC'_{c' \leqslant0}$. Par hypothèse et par orthogonalité (faible), $\Hom_\CC(N,D(P))=0$ pour tout objet $P\in \CC'_{c' \geqslant 1}$, ce qui donne par adjonction $\Hom_{\CC'}(G(N),P)=0$ pour tout $P\in \CC'_{c' \geqslant 1}$ ce qui implique par orthogonalité (forte) que $G(N)\in \CC'_{c'\leqslant0}$. 

%De même pour l'énoncé dual.
%\end{proof}

Les théorèmes \ref{thm.constr} et \ref{thm.constr2} ainsi que la description des structures de poids qu'ils décrivent permettent d'établir les corollaires suivants.

\begin{corollaire}\label{CORRRRRRR}
Supposons les conditions du théorème \ref{thm.constr2} satisfaites ainsi que la condition $(ii)$ de $loc. cit.$. Notons respectivement
\begin{center}
$(a_\infty)$. $\CC'=\CA{\H}$,\hspace{2cm}$(b_\infty)$. $\CC'=\CAep{\H}$.
\end{center}
Alors il existe des structures de poids $w/\CC$ et $w'/\CC'$ telle que $w'=\restr{w}{\CC'}$. 
\end{corollaire}

\begin{corollaire}\label{Pour.Levine}
Supposons que l'une des conditions $(a)$, $(b)$, $(a_\infty)$ ou $(b_\infty)$ des théorèmes \ref{thm.constr} et \ref{thm.constr2} soit satisfaite (avec les conditions qui s'imposent sur $\CC$ et $\H$). Supposons également que $\H$ satisfasse la condition $(ii)$ de \textit{loc. cit.}. Soit $\H'\subset\H$.
Notons respectivement
\begin{center}
$(a)$. $\CC'=\CA{\H'}$,\hspace{2.27cm}$(a_\infty)$. $\CC'=\CAinf{\H'}$,\\
$(b)$. $\CC'=\CAep{\H'}$,\hspace{2cm}$(b_\infty)$. $\CC'=\CAepinf{\H'}$.
\end{center}
Alors il existe des structures de poids $w/\CC$ et $w'/\CC'$ telles que $w'=\restr{w}{\CC'}$.
\end{corollaire}

\section{Les motifs de Beilinson en dix leçons.}
\addcontentsline{toc}{section}{\protect\numberline{\thesection} Les motifs de Beilinson en dix leçons.}

Dans la suite on choisit de se placer dans la catégorie des \SW{motifs de Beilinson} (\cite[déf. 13.2.1]{CD})
$$\DMB{S}$$
où $S$ désigne un schéma de base (de type fini au dessus de $B$ ; \cf introduction). 
Elle peut se définir à partir des faisceaux étales (sur le site $(Sm/S)_{\text{\'et}}$) à coefficients rationnels (\cite[thm. 15.2.16]{CD}) : on considère la catégorie dérivée de cette catégorie de faisceaux. Dans cette catégorie on veut identifier $X$ à $\AAA^1_X$. Ce procédé s'appelle la $\AAA^1$-localisation (\cite[déf. 5.2.16]{CD}). Avec cette localisation on obtient la catégorie ``effective'' des motifs de Beilinson (\cite[ex. 5.2.17]{CD}); cette catégorie effective est monoïdale symétrique (\cite[prop. 5.2.2]{CD}). Pour arriver à $\DMB{S}$ on inverse (pour le produit tensoriel) le twist de Tate, noté $\Un_S(1)$ (\cite[déf. 5.3.22, ex. 5.3.34]{CD}).

Pour $S=\Spec(k)$ (et de manière générale lorsque $S$ est géométriquement unibranche) il existe une équivalence de catégorie entre la catégorie des motifs de Beilinson et la catégorie des motifs à la Voevodsky (construit avec les faisceaux avec transferts) à coefficients rationnels (\cite[thm. 15.1.4]{CD}).

Voici une liste des propriétés  de la catégorie des motifs de Beilinson, $f:\morph{S}{T}$ désignant un morphisme de schémas : 
\begin{description}
\item[1.] On a les six opérations de Grothendieck : issu du foncteur de restriction, on a le foncteur $f^* : \morph{\DMB{T}}{\DMB{S}}$
qui admet un adjoint à droite $f_*$. Par exemple, en notant $\Un_S$ l'unité pour le produit tensoriel (issu du faisceau constant sur $S$ qui vaut $\Q$), on a $f^*\Un_T=\Un_S$. Dans le cas ou $f$ est lisse, $f^*$ admet également un adjoint à gauche
$f_\sharp : \morph{\DMB{S}}{\DMB{T}}$
(issu du foncteur d'oubli de la base). Partant du foncteur de prolongement par zéro, on a 
$f_! : \morph{\DMB{S}}{\DMB{T}}$
qui admet un adjoint à droite $f^!$. En particulier si $f$ est propre $f_!=f_*$ (\cite[thm. 2.2.14.$(1)$]{CD}). La catégorie $\DMB{S}$ est monoïdale symétrique fermée ; on notera $\otimes_S$ le produit tensoriel et $\HHom_S$ son adjoint à droite. \`A noter enfin la formule de projection (\cite[thm. 2.4.21.$v$]{CD}) : pour tout $M\in \DMB{S}, N\in \DMB{T}$, $f_!M\otimes_TN=f_!(M\otimes_Sf^*N)$.
\item[2.] On a des isomorphismes de changement de base. Précisement, pour tout carré cartésien
$$\xymatrix{X'\cartesien\ar[d]_{\beta'}\ar[r]^{\alpha'}&Y'\ar[d]^\beta\\
X\ar[r]_\alpha&Y}$$
on a $\beta^*\alpha_!=\alpha'_!\beta'^*$ et $\beta'_*\alpha'^!=\alpha^!\beta_*$ (\cite[thm. 2.2.14.$(4c)$]{CD}).
\item[3.] Si $f$ est lisse de dimension relative $d$ on a un isomorphisme de pureté relative (\cite[thm. 2.4.15.$(iii)$, rm. 2.4.16]{CD}): 
$$f^!\Un_T=f^*\Un_T(d)[2d]=\Un_S(d)[2d], \qquad f_!\Un_T=f_\sharp\Un_T(-d)[-2d].$$
\item[4.] Si $f$ est une immersion fermée de codimension $c$ entre schémas réguliers on a un isomorphisme de pureté absolue (\cite[thm. 13.4.1]{CD}):
$$f^!\Un_T=\Un_S(-c)[-2c]$$
\item[5.] Si $U$ est un ouvert de $S$ de fermé complémentaire $Z$, alors en notant $j:\immop{U}{S}$ et $i:\immcl{Z}{S}$ les immersions canoniques, on a le triangle distingué de localisation (\cite[prop. 2.3.3.$(2)$, thm. 2.2.14.$(2)$]{CD})
$$\tridis{j_!\Un_U}{\Un_S}{i_!\Un_Z}$$
\item[6.] On a la $h$-descente : considérons le diagramme suivant, où les carrés sont cartésiens
$$\xymatrix{
Z'\immfer[r]\ar@/_/[rd]^a\ar[d]&T'\ar[d]^p&\immouv[l]\ar[d]U'\\
Z\immfer[r]_i&T&\immouv[l]U}$$
où $p$ est une altération de Galois de groupe $G$ telle que génériquement $\morph{T'/G}{T}$ est fini, surjectif et radiciel, $U$ est normal et $\morph{U'}{U}$ est fini alors on a le triangle distingué pour tout $M\in \DMBc{T}$ (\cite[thm. 14.3.7]{CD})
$$\tridis{M}{i_!i^*M\oplus \left(p_!p^*M\right)^G}{\left(a_!a^*M\right)^G}.$$
\item[7.] Si $S$ est régulier on a (\cite[cor. 13.2.14]{CD}) $$\forall (a,b)\in \Z^2,\;\Hom_{\DMB{S}}(\Un_S,\Un_S(a)[b])=\Gr_\gamma^a\KK_{2a-b}(S)_\Q,$$
où $\Gr_\gamma$ désigne le gradué pour la filtration $\gamma$ (\cite[\S 13.1]{CD}) et $\KK_n(S)_\Q:=\KK_n(S)\otimes_\Z\Q$ la $\KK$-théorie rationnelle de Quillen qui est nulle si $n<0$.
\item[8.] Lorsque $f$ est lisse, on pose $M_T(S):=f_\sharp\Un_S$ ; c'est le \SW{motif associé} à $S$. La catégorie des \SW{motifs constructibles} (\cite[déf. 1.4.7]{CD}) est $\DMBc{T}:=\CAep{\G_T}$ où  
$$\DMB{T}\supset\G_T:=\left\{M_T(S)(n)\big|n\in \Z, f:\morph{S}{T}\hbox{ lisse}\right\}.$$  La catégorie $\DMBc{T}$ correspond à la sous-catégorie pleine de $\DMB{T}$ formée des objets compacts $\DMB{T}_c$ (\cite[cor. 5.2.37]{CD}). \`A noter de plus que $\DMB{S}=\CAepinf{\DMBc{S}}$.
\item[9.] Les six opérations de Grothendieck respectent les objets constructibles (\cite[thm. 14.1.31]{CD}).
\item[10.] Les catégories $\DMB{S}$ et $\DMBc{S}$ sont pseudo-abéliennes : par construction $\DMB{S}$ est une catégorie triangulée admettant des petites sommes (voir  par exemple \cite[prop. 1.6.8]{Nee}). De même, par construction, la catégorie $\DMBc{S}$ est épaisse.
\end{description}

\begin{remarque}
A noter que le lecteur pourra également choisir de se placer dans la catégorie ${\mathbf{S}\mathbf{H}}_{\mathfrak{M}}$ (\cf \cite[déf. 4.5.52, 4.2.21]{Ayoub} avec $\mathfrak{M}$ la catégorie des $\Q$-espaces vectoriels ; la topologie étant la topologie étale). D'après \cite[thm. 15.2.16]{CD} celle-ci est équivalente à $\DMB{S}$. Une majeur partie  des propriétés précédentes est d'ailleurs prouvée intrinsèquement dans \cite{Ayoub}.
\end{remarque}

\section{Structure de poids et motifs.}
\addcontentsline{toc}{section}{\protect\numberline{\thesection} Structure de poids et motifs.}
Dans cette partie nous allons déterminer une structure de poids sur la catégorie des motifs de Beilinson et par restriction sur la catégorie des motifs de Beilinson constructibles. Pour cela nous allons utiliser les théorèmes de construction \ref{thm.constr} et \ref{thm.constr2}. Dans les deux cas il s'agit d'exhiber une catégorie satisfaisant la condition d'orthogonalité $(ii)$ de \textit{loc. cit.}. Le théorème clef est le suivant.

La notation $(rap.\; i)$ fait référence au rappel numéro $i$ de la section précédente.

\begin{lemme}[Lemme de Chow motivique]\label{Mot.Chow.lem}
Soit $p:\morph{X}{S}$ un morphisme propre à domaine régulier. Il existe un morphisme $\pi:\morph{X_0}{X}$ projectif à domaine régulier tel que 
\begin{description}
\item[$(i)$.] le morphisme composé $p\pi$ est projectif,
\item[$(ii)$.] le schéma $X_0$ a la même dimension que le schéma $X$,
\item[$(iii)$.] le motif $p_!\Un_X$ est un rétracte de $(p\pi)_!\Un_{X_0}$.
\end{description}
\end{lemme}

\begin{proof}
Appliquons le lemme de Chow au morphisme $p$ (\cf \cite[cor. 5.6.2]{EGA2}) : on trouve un morphisme $\pi':\morph{X'_0}{X}$ projectif tel que $p\pi'$ est projectif et tel que $X'_0$ a la même dimension que $X$. On considère une altération de Galois $\morph{X_0}{X'_0}$ et on note $\pi:\morph{X_0}{X}$ le morphisme composé ;
\begin{multicols}{3}
\noindent c'est un morphisme projectif, c'est à dire composé d'une immersion fermée et d'un morphisme lisse,
\columnbreak

$\xymatrix@R=0.2cm@C=1cm{X_0\ar[rr]^\pi\immfer[rd]_i&&X\\
&Y\ar[ru]_s&}$\\\\
via $(rap.\; 3)$ et $(rap.\; 4)$, on a
\columnbreak
\begin{eqnarray*}
\pi^!\Un_{X}&=& i^!s^!\Un_X\\
&=& i^!\Un_{Y}(c)[2c]\\
&=& \Un_{X_0}(c-c)[2c-2c]\\
&=& \Un_{X_0}.
\end{eqnarray*}
\end{multicols}

Considérons alors la composé d'adjonction $\alpha : \morphp{\Un_X}{\pi_*\pi^*\Un_X=\pi_!\Un_{X_0}=\pi_!\pi^!\Un_X}{\Un_X}$. On observe que $\alpha=d\cdot\Id_{\Un_X}$.
En effet, sans perdre en généralité, on peut supposer que $X$ et $X_0$ sont connexes ; pour n'importe quelle immersion ouverte $j:\morph{U}{X}$ le morphisme
$j^*:\morph{\Hom_{\DMBc{X}}(\Un_X,\Un_X)}{\Hom_{\DMBc{U}}(\Un_U,\Un_U)}$ est un isomorphisme (car $\Hom_{\DMBc{X}}(\Un_X,\Un_X)=\Hom_{\DMBc{U}}(\Un_U,\Un_U)=\Q$, \cf \cite[prop. 10.2.11.$(1)$]{CD}) de sorte que le changement de base propre $(rap.\; 2)$ permet de se ramener à prouver que $\alpha=d\cdot\Id_{\Un_X}$ pour n'importe quelle restriction de $\pi$.
\begin{multicols}{2}
$$\xymatrix{U_0\immouv[r]\ar[d]_{\pi_U}&X_0\ar[d]^\pi\\
U\immouv[r]_j&X}$$

\columnbreak

\noindent On peut trouver un ouvert $U$ de $X$ tel que $\pi_U$ est plat, fini de degré $d=[\kappa(X_0):\kappa(X)]$ ($\kappa(X)$ désignant le corps des fonctions de $X$ ; ceci est possible éssentielement parce que $\pi$ induit un morphisme plat fini de degré $d$ entre $\Spec(\kappa(X_0))$ et $\Spec(\kappa(X))$). La conclusion suit de \cite[prop. 12.7.6]{CD}.
\end{multicols}
Puisque les coéfficients sont rationnels, $d$ est inversible et donc $\Un_X$ s'identifie à un rétracte de $\pi_!\Un_{X_0}$. On conclut en composant par $p_!$ (qui est un foncteur additif).
\end{proof}

\begin{theoreme}\label{thm.clef}
Soit $f:\morph{T}{Y}$ un morphisme de schémas tel que $Y$ soit régulier. Alors 
$$\forall (a,b)\in \Z^2,\; b>2a,\quad \Hom_{\DMB{Y}}(f_!\Un_T,\Un_Y(a)[b])=0.$$
\end{theoreme}

\begin{proof}$ $
\begin{description}
\item[\'ETAPE 1 : L'énoncé est vrai pour les immersions fermées entre schémas réguliers.] Dans \\ ce cas on a
\begin{eqnarray*}
\Hom_{\DMB{Y}}(f_!\Un_T,\Un_Y(a)[b])&=&\Hom_{\DMB{T}}(\Un_T,f^!\Un_Y(a)[b])\\
&\overset{(rap.\; 4)}{\underset{a'=a-c,\ b'=b-2c}{=}}&\Hom_{\DMB{T}}(\Un_T,\Un_T(a')[b'])\\
&\overset{(rap.\; 7)}{=}&\Gr_\gamma^{a'}\KK_{2a'-b'}(T)_\Q\\
&\underset{2a'-b'<0}{=}&0.
\end{eqnarray*}
\item[\'ETAPE 2 : On peut supposer $T$ régulier.] Nous allons raisonner par récurrence sur la dimension de $T$. Pour cela on considère une altération de Galois comme dans $(rap.\; 6)$, qui existe en vertue de \cite[thm. 14.3.6]{CD}, avec $M=\Un_T$ pour obtenir le triangle distingué
$\tridis{\Un_T}{i_!\Un_Z\oplus \left(p_!\Un_{T'}\right)^G}{\left(a_!\Un_{Z'}\right)^G}$, où $T'$ est régulier.
En le composant par $f_!$  et en décalant on aboutit à 
$$\tridis{f_!\left(a_!\Un_{Z'}\right)^G[-1]}{f_!\Un_T}{(fi)_!\Un_Z\oplus f_!\left(p_!\Un_{T'}\right)^G}$$
On applique le foncteur cohomologique $\Hom_{\DMB{Y}}(\bullet,\Un_Y(a)[b])$ pour obtenir la suite exacte 
$$\xymatrix@R=0.5cm{
\Hom_{\DMB{Y}}((fi)_!\Un_Z,\Un_Y(a)[b])\times\Hom_{\DMB{Y}}(f_!\left(p_!\Un_{T'}\right)^G,\Un_Y(a)[b])\ar@{=}[d]\\
\Hom_{\DMB{Y}}((fi)_!\Un_Z\oplus f_!\left(p_!\Un_{T'}\right)^G,\Un_Y(a)[b])\ar[d]\\
\Hom_{\DMB{Y}}(f_!\Un_T,\Un_Y(a)[b])\ar[d]\\
\Hom_{\DMB{Y}}(f_!\left(a_!\Un_{Z'}\right)^G[-1],\Un_Y(a)[b])\ar@{=}[d]\\
\Hom_{\DMB{Y}}(f_!\left(a_!\Un_{Z'}\right)^G,\Un_Y(a)[b+1])\\}$$
La conclusion suit de l'hypothèse de récurrence et du fait que $T'$ soit régulier.
\item[\'ETAPE 3 : L'énoncé est vrai pour les morphismes projectifs.] $ $
\setlength{\columnseprule}{1pt} 
\begin{multicols}{2}
D'après l'étape $2$, on peut supposer que $T$ est régulier (dans la preuve de l'étape 2, les morphismes $p$, $i$ et $a$ sont projectifs, de sorte que l'on ne change pas la nature du morphisme $f$). 
On a une factorisation en une immersion fermée et un morphisme lisse, 
où $P$ est régulier (car $s$ est lisse). On obtient alors :
\columnbreak
$$\xymatrix@R=1.5cm@C=2cm{T\ar[rr]^f\immfer[dr]_c&&Y\\
&P\ar[ru]_s&}$$
\end{multicols}
\vspace{-1cm}
\begin{eqnarray*}
\Hom_{\DMB{Y}}(f_!\Un_T,\Un_Y(a)[b])&=&\Hom_{\DMB{Y}}(s_!c_!\Un_T,\Un_Y(a)[b])\\
&=&\Hom_{\DMB{P}}(c_!\Un_T,s^!\Un_Y(a)[b])\\
&\overset{(rap.\; 3)}{\underset{a'=a+d,\ b'=b+2d}{=}}&\Hom_{\DMB{P}}(c_!\Un_T,\Un_P(a')[b'])\\
&\overset{\text{\'Etape }1}{=}&0.
\end{eqnarray*}

\item[\'ETAPE 4 : L'énoncé est vrai pour les morphismes propres.] On se ramène au cas projectif par le lemme de Chow motivique.
\item[\'ETAPE 5 : Conclusion.] On choisit une compactification de $T$ (voir par exemple \cite[\S 4 thm. 2]{Nagata}) : 
\setlength{\columnseprule}{1pt} 
\begin{multicols}{2}
$$\xymatrix@R=1.5cm@C=2cm{T\ar[rd]_f\immouv[r]^j&\overline{T}\ar[d]^p&\partial\overline{T}\immfer[l]_i\ar[ld]^g\\
&Y&}$$

\columnbreak
On utilise 
$\tridis{j_!\Un_T}{\Un_{\overline{T}}}{i_!\Un_{\partial\overline{T}}}$
que l'on compose par $p_!$ et que l'on décale :
$$\tridis{g_!\Un_{\partial\overline{T}}[-1]}{f_!\Un_T}{p_!\Un_{\overline{T}}}$$
On applique $\Hom_{\DMB{Y}}(\bullet,\Un_Y(a)[b])$ pour conclure (via l'étape $5$ ; les morphismes $p$ et $g$ sont propres).
\end{multicols}
\end{description}
\end{proof}

\begin{theoreme}\label{cor.constr.WS.motifs}
Soit $$\DMB{S}\supset\H_S:=\left\{f_!\Un_X(x)[2x]\big|x\in \Z,\;f:\morph{X}{S} \textrm{ propre à domaine régulier}\right\}.$$
\begin{description}
\item[$(i)$. ] Il existe une unique structure de poids $W/\DMB{S}$ telle que $\H_S^\infty\subset\DMBe{S}$.
\item[$(ii)$. ] Il existe une unique structure de poids $w/\DMBc{S}$ telle que $\H_S\subset\DMBce{S}$. Précisément $\DMBce{S}=\Kar(\H_S^\oplus)$.
\item[$(iii)$. ] $w=\restr{W}{\DMBc{S}}$.
\end{description}
\end{theoreme}

\begin{proof}
On applique \ref{CORRRRRRR}$.(b_\infty)$.
\begin{description}
\item[$\bullet$] La catégorie $\H_S$ engendre $\DMBc{S}$ : \cite[cor. 14.3.9]{CD}. Donnons, pour le confort du lecteur, une idée de la preuve. Le foncteur $f_!$ respectant les objets constructibles $(rap.\; 9)$, on a $\H_S\subset\DMBc{S}$. Pour conclure, il suffit de voir que $\G_S\subset\CAep{\H_S}$ c'est à dire que $f_\sharp\Un_X(n)$ pour $f:\morph{X}{S}$ un morphisme lisse et $n\in \Z$, est dans $\CAep{\H_S}$. Mais $(rap.\; 3)$ permet de passer de $\sharp$ à $!$, le principe de l'étape $5$ de la preuve du théorème précédent permet de se ramener au cas propre et le principe de l'étape $2$ permet de se ramener au cas où le domaine est régulier. 

Ainsi $\CAep{\H_S}=\DMBc{S}$, ce qui implique par $(rap.\; 8)$,  $\CAepinf{\H_S}=\DMB{S}$.
\item[$\bullet$] Il faut voir que si $H_1$ et $H_2$ sont des objets de $\H_S$ et que $i\in \N_{>0}$ alors $\Hom_{\DMB{S}}(H_1,H_2[i])=0$. Mais de tels objets sont de la forme $f_!\Un_X(x)[2x]$, pour $f$ propre à domaine régulier et $x\in \Z$. On se ramène à calculer $\Hom_{\DMB{S}}(f_!\Un_X,g_!\Un_Y(a)[b])$ lorsque $b>2a$.
\begin{multicols}{2}
\begin{flushright}
$\Hom_{\DMB{S}}(f_!\Un_X,g_!\Un_Y(a)[b])$
\hspace*{-1.1cm}
\end{flushright}
$$\xymatrix@C=6cm@R=1.4cm{T\ar[r]^{f'}\ar[d]_{g'}\cartesien&Y\ar[d]^g\\
X\ar[r]_f&S}$$
\begin{eqnarray*}
&\overset{(rap.\; 1)}{=}&\Hom_{\DMB{S}}(f_!\Un_X,g_*\Un_Y(a)[b]) \\
&=&\Hom_{\DMB{X}}(\Un_X,f^!g_*\Un_Y(a)[b]) \\
&\overset{(rap.\; 2)}{=}&\Hom_{\DMB{X}}(\Un_X,g'_*f'^!\Un_Y(a)[b]) \\
&=&\Hom_{\DMB{T}}(\Un_T,f'^!\Un_Y(a)[b]) \\
&=&\Hom_{\DMB{Y}}(f'_!\Un_T,\Un_Y(a)[b])\\
&\overset{\ref{thm.clef}}{=}&0.
\end{eqnarray*}
\end{multicols}
\end{description}
La détermination exacte du c{\oe}ur suit du théorème \ref{thm.constr}.
\end{proof}

\begin{remarque}
On arrive au même résultat lorsqu'on demande que les objets de $\H_S$ proviennent de morphismes projectifs à domaine régulier. 
Ceci prouve en particulier que $W$ (resp. $w$) coïncide avec la structure de poids notée $w_{Chow}^{big}$ (resp. $w_{Chow}$) dans \cite[thm. 2.2.1]{Bo2}.

Lorsque $S=\Spec(k)$, $k$ désignant un corps de caractéristique $0$, on retrouve la structure de poids de \cite[\S 6.5]{Bo}.
\end{remarque}

\begin{remarque}\label{rm.importante}
Considérons les catégories suivantes
$$\DMBc{S}\supset\Neg_S:=\left\{f_!\Un_X(a)[b]\big|(a,b)\in \Z^2,\;b\leqslant 2a,\;f:\morph{X}{S} \textrm{ propre à domaine régulier}\right\}^\oplus.$$
$$\DMBc{S}\supset\Pos_S:=\left\{f_!\Un_X(a)[b]\big|(a,b)\in \Z^2,\;b\geqslant 2a,\;f:\morph{X}{S} \textrm{ propre à domaine régulier}\right\}^\oplus.$$
Alors par construction (\cf preuve de \ref{thm.constr} et \ref{thm.constr2} ainsi que les remarques \ref{rm.thm.constr} et \ref{rm.thm.constr2})
$$\DMBn{S}=\Kar\left(\SE_{\DMB{S}}^\infty\left(\Kar(\Neg_S)\right)\right),\qquad 
\DMBp{S}=\Kar\left(\SE_{\DMB{S}}^\infty\left(\Kar(\Pos_S)\right)\right).$$
$$\DMBcn{S}=\Kar\left(\SE_{\DMBc{S}}\left(\Kar(\Neg_S)\right)\right),\qquad \DMBcp{S}=\Kar\left(\SE_{\DMBc{S}}\left(\Kar(\Pos_S)\right)\right).$$
De plus l'orthogonalité forte, la remarque \ref{rem...}, le lemme \ref{orth.ext.stab} et le lemme \ref{orth.ext.stab.inf} nous donnent
$$\DMBp{S}={^\bot}\Neg_S[-1],\qquad \DMBcn{S}=\Pos_S[1]^\bot,\qquad\DMBcp{S}={{}^\bot}\Neg_S[-1].$$
On prendra garde que l'orthogonal de la première égalité se calcule dans $\DMB{S}$ tandis que les autres orthogonaux se calculent dans $\DMBc{S}$.
\end{remarque}

\begin{lemme}\label{proposition.negative}
Notons
$$\DMBc{S}\supset\G_S^-:=\left\{f_\sharp\Un_X(a)[b]\big|(a,b)\in \Z^2,\;b\leqslant 2a,\;f:\morph{X}{S} \textrm{ lisse}\right\}.$$
Alors on a $\Kar\left(\SE_{\DMBc{S}}\left(\Kar(\G_S^-)\right)\right)\subset\DMBcn{S}$.
\end{lemme}

\begin{proof}
D'aprés la remarque précédente, il suffit de voir que $\G_S^-$ est orthogonale à $\Pos[1]$ (où l'orthogonal est pris dans $\DMBc{S}$); il s'agit donc de prouver que 
$$\Hom_{\DMB{S}}(f_\sharp\Un_X,g_!\Un_Y(a)[b])=0$$ lorsque $a$ et $b$ sont des entiers tels que $b>2a$, $f:\morph{X}{S}$ est un morphisme lisse et $g:\morph{Y}{S}$ est un morphisme propre à domaine régulier. Mais 
\begin{multicols}{2}
\begin{flushright}
$\Hom_{\DMB{S}}(f_\sharp\Un_X,g_!\Un_Y(a)[b])$
\hspace*{-1.3cm}
\end{flushright}
$$\xymatrix@C=6cm@R=2cm{T\ar[r]^{f'}\ar[d]_{g'}\cartesien&Y\ar[d]^g\\
X\ar[r]_f&S}$$
\begin{eqnarray*}
&=&\Hom_{\DMB{X}}(\Un_X,f^*g_!\Un_Y(a)[b]) \\
&\overset{(rap.\; 2)}{=}&\Hom_{\DMB{X}}(\Un_X,g'_!f'^*\Un_Y(a)[b]) \\
&=&\Hom_{\DMB{X}}(\Un_X,g'_!\Un_T(a)[b]) \\
&=&\Hom_{\DMB{X}}(\Un_X,g'_*\Un_T(a)[b]) \\
&=&\Hom_{\DMB{T}}(g'^*\Un_X,\Un_T(a)[b]) \\
&=&\Hom_{\DMB{T}}(\Un_T,\Un_T(a)[b]) \\
&\overset{\ref{thm.clef}}{=}&0.
\end{eqnarray*} 
\end{multicols}
Le schéma $T$ est régulier car $Y$ est régulier et $f'$ est lisse.
\end{proof}

\`A présent nous allons établir les relations de $w$-exactitude des six opérations de Grothendieck. 

\begin{theoreme}\label{propo.compatibilite}
Soit $\alpha:\morph{S}{T}$ un morphisme de schémas.
\begin{description}
\item[$(i)$. ] Les foncteurs
$\alpha^* : \morph{\DMB{T}}{\DMB{S}}$ et $\alpha_! : \morph{\DMB{S}}{\DMB{T}}$
sont $w$-exactes à gauche.
\item[$(i')$. ] Les foncteurs
$\alpha_* : \morph{\DMB{S}}{\DMB{T}}$ et $\alpha^! : \morph{\DMB{T}}{\DMB{S}}$
sont $w$-exactes à droite.
\item[$(i_c)$. ] Les foncteurs
$\alpha^* : \morph{\DMBc{T}}{\DMBc{S}}$ et $\alpha_! : \morph{\DMBc{S}}{\DMBc{T}}$
sont $w$-exactes à gauche.
\item[$(i'_c)$. ] Les foncteurs
$\alpha_* : \morph{\DMBc{S}}{\DMBc{T}}$ et $\alpha^! : \morph{\DMBc{T}}{\DMBc{S}}$
sont $w$-exactes à droite.
\item[$(ii)$. ] Supposons que $\alpha$ soit lisse, alors le foncteur
$\alpha_\sharp : \morph{\DMB{S}}{\DMB{T}}$
est $w$-exacte à gauche.
\item[$(ii')$. ] Supposons que $\alpha$ soit lisse, alors le foncteur
$\alpha^* : \morph{\DMB{T}}{\DMB{S}}$
est $w$-exacte.
\item[$(ii_c)$. ]Supposons que $\alpha$ soit lisse, alors le foncteur
$\alpha_\sharp : \morph{\DMBc{S}}{\DMBc{T}}$
est $w$-exacte à gauche.
\item[$(ii'_c)$. ] Supposons que $\alpha$ soit lisse, alors le foncteur
$\alpha^* : \morph{\DMBc{T}}{\DMBc{S}}$
est $w$-exacte.
\item[$(iii)$. ] Soit $(n,n')\in \Z^2$. Le bifoncteur $\otimes_S : \morph{\DMB{S}\times\DMB{S}}{\DMB{S}}$ induit un bifoncteur\\ $\morph{\DMB{S}_{W\leqslant n}\times\DMB{S}_{W\leqslant n'}}{\DMB{S}_{W\leqslant n+n'}}$.
\item[$(iii')$. ] Soit $(n,p)\in \Z^2$. Le bifoncteur $\HHom_S:\morph{\DMB{S}^\opp\times\DMB{S}}{\DMB{S}}$ induit un bifoncteur $\morph{\DMB{S}^{\opp}_{W\leqslant n}\times\DMB{S}_{W\geqslant p}}{\DMB{S}_{W\geqslant p-n}}$.
\item[$(iii_c)$. ] Soit $(n,n')\in \Z^2$. Le bifoncteur $\otimes_S:\morph{\DMBc{S}\times\DMBc{S}}{\DMBc{S}}$ induit un bifoncteur $\morph{\DMBc{S}_{w\leqslant n}\times\DMBc{S}_{w\leqslant n'}}{\DMBc{S}_{w\leqslant n+n'}}$.
\item[$(iii'_c)$. ] Soit $(n,p)\in \Z^2$. Le bifoncteur $\HHom_S:\morph{\DMBc{S}^\opp\times\DMBc{S}}{\DMBc{S}}$ induit un bifoncteur $\morph{\DMBc{S}^{\opp}_{w\leqslant n}\times\DMBc{S}_{w\geqslant p}}{\DMBc{S}_{w\geqslant p-n}}$.
\item[$(iv)$. ] Pour tout entier $n\in \Z$, le foncteur $\bullet\otimes_S\Un_S(n)[2n]:\morph{\DMB{S}}{\DMB{S}}$ est $w$-exacte.
\item[$(iv_c)$. ] Pour tout entier $n\in \Z$, le foncteur $\bullet\otimes_S\Un_S(n)[2n]:\morph{\DMBc{S}}{\DMBc{S}}$ est $w$-exacte.
\item[$(v)$. ] On a toujours $\Un_S\in \G_S^-\subset\DMBcn{S}$. De plus si $S$ est régulier alors $\Un_S\in\H_S\subset \DMBce{S}$.
\end{description}
\end{theoreme}

\begin{proof}
Le morphisme $\Id_S:\morph{S}{S}$ est lisse donc $\Un_S\in\G_S^-$. Si de plus $S$ est régulier alors $\Id_S$ est propre à domaine régulier donc $\Un_S\in \H_S$ ce qui prouve $(v)$. 

Soit $?\in \{i, ii, iii\}$. On démontre (\cf \cite[prop. 1.2.3.9]{Bo2}) qu'un adjoint à gauche est $w$-exacte à gauche si et seulement si l'adjoint à droite associé est $w$-exacte à droite ; ainsi l'énoncé $(?)$ (resp. $(?_c)$) équivaut à $(?')$ (resp. $(?'_c)$). L'énoncé $(?_c)$ (resp. $(?'_c)$) se déduit de $(?)$ (resp. $(?')$) par \ref{cor.constr.WS.motifs}$.(iii)$ et $(rap.\; 9)$. En conclusion, il suffit de montrer $(i')$,  $(ii')$, $(iii)$ et $(iv)$.
\begin{description}
\item[$(i')$. ] Soit $P\in \DMBp{S}$ ; on veut montrer que $\alpha_*P\in\DMBp{T}$. D'aprés la remarque \ref{rm.importante}, il suffit de voir que pour tout $f_!\Un_X(a)[b]\in \Neg_T$, $\Hom_{\DMB{T}}(f_!\Un_X(a)[b],\alpha_*P[1])=0$.\vspace{0cm}
\begin{multicols}{2}
\begin{flushright}
$\Hom_{\DMB{T}}(f_!\Un_X(a)[b],\alpha_*P[1])$
\hspace*{-1.2cm}
\end{flushright}
$$\xymatrix@C=5cm@R=0.5cm{Y\ar[d]_{f'}\ar[r]^{\alpha'}\cartesien&X\ar[d]^f\\
S\ar[r]_\alpha&T}$$
\begin{eqnarray*}
&=&\Hom_{\DMB{S}}(\alpha^*f_!\Un_X(a)[b],P[1])\\
&\overset{(rap.\; 2)}{=}&\Hom_{\DMB{S}}(f'_!\alpha'^*\Un_X(a)[b],P[1])\\
&=&\Hom_{\DMB{S}}(f'_!\Un_{Y}(a)[b],P[1])\\
&=&0.
\end{eqnarray*}
\end{multicols}
Pour la dernière égalité : en utilisant l'argument de l'étape $2$ de la preuve de \ref{thm.clef} (on applique le foncteur $\Hom_{\DMB{S}}(\bullet,P[1])$ pour la conclusion), on peut supposer que $Y$ est régulier. Dans ce cas $f'_!\Un_{Y}(a)[b]\in \Neg_S\subset\DMBn{S}=\DMBp{S}[1]^\bot$.
Pour le second foncteur on raisonne comme précédement : il suffit de montrer que pour tout $P\in \DMBp{T}$ et $f_!\Un_X(a)[b]\in \Neg_S$ on a $\Hom_{\DMB{S}}(f_!\Un_X(a)[b],\alpha^!P[1])=0$. Par adjonction il revient au même de montrer que $\Hom_{\DMB{T}}((\alpha f)_!\Un_X(a)[b],P[1])=0$. En utilisant le principe de l'étape $6$ (\cf \ref{thm.clef}), on peut supposer $\alpha f$ propre (et $X$ est toujours régulier ; l'étape 5 permet même de se ramener au cas où $f$ est quasi-projectif). Ainsi $(\alpha f)_!\Un_X(a)[b]\in \Neg_T\subset \DMBn{T}=\DMBp{T}[1]^\bot$.

\item[$(ii')$. ] On montre que $\alpha^*$ est $w$-exacte à droite. Soit $P\in \DMBp{T}$ ; comme pour le cas $(i)'$, il suffit de montrer que pour tout $f_!\Un_X(a)[b]\in \Neg_S$ on a $\Hom_{\DMB{S}}(f_!\Un_X(a)[b],\alpha^*P[1])=0$. Par adjonction il revient au même de montrer que $\Hom_{\DMB{T}}(\alpha_\sharp f_!\Un_X(a)[b],P[1])=0$. En utilisant la pureté relative $(rap.\; 3)$, on peut remplacer le symbole $\sharp$ par $!$ ; dans ce cas $a$ est remplacé par $a'=a+d$ et $b$ par $b'=b+2d$, où $d$ est la dimension relative de $\alpha$. D'après le point $(v)$ déjà prouvé, $\Un_X(a')[b']\in \G_X^-\subset\DMBcn{X}$. Par le point $(i)$, $(\alpha f)_!\Un_X(a')[b']\in \DMBn{T}$. La conclusion suit par orthogonalité.

\item[$(iv)$. ] Soit $P\in \DMBp{S}$. On va montrer que $P(n)[2n]\in \DMBp{S}$. Il suffit de voir que pour tout $f_!\Un_S(a)[b]\in \Neg_S$, $\Hom_{\DMB{S}}(f_!\Un_S(a)[b],P(n)[2n+1])=0$. Or ce groupe s'identifie à $\Hom_{\DMB{S}}(f_!\Un_S(a-n)[b-2n],P[1])$ et $f_!\Un_S(a-n)[b-2n]\in \Neg_S$.
Soit $N\in \DMBn{S}$. On va montrer que $N(n)[2n]\in \DMBn{S}$. Il suffit de voir que pour tout $P\in \DMBp{S}$, $\Hom_{\DMB{S}}(N(n)[2n],P[1])=0$. Or ce groupe s'identifie à $\Hom_{\DMB{S}}(N,P(-n)[-2n+1])$ et le raisonement précédent donne $P(-n)[-2n]\in \DMBp{S}$.

\item[$(iii)$. ] Soient $f_!\Un_X(a)[b]\in \Neg_S$, $N\in \DMBn{S}$ et $P\in \DMBp{S}$ alors, utilisant la formule de projection rappelée en $(rap.\; 1)$, on a $\Hom_{\DMB{S}}(f_!\Un_X(a)[b]\otimes_SN,P[1])=\Hom_{\DMB{S}}(f_!(\Un_X(a)[b]\otimes_Xf^*N),P)=\Hom_{\DMB{X}}(f^*N(a)[b],f^!P[1])=0$. Les points $(i)$ et $(iv)$ justifient que $f^*N(a)[b]\in \DMBn{X}$, $(i')$ justifie $f^!P\in \DMBp{X}$ ; la dernière égalité suit par orthogonalité. On a ainsi montré que, pour tout $N\in \DMBn{S}$, $\bullet\otimes_S N:\morph{\DMB{S}}{\DMB{S}}$ transforme les objets de $\Neg_S$ en objet de $\DMBp{S}[1]^\bot=\DMBn{S}$. On peut sans peine remplacer $\Neg_S$ par ses retractes, ses extensions et des sommes arbitraires reconstruisant ainsi $\DMBn{S}$. Nous avons ainsi montré que le produit tensoriel induit un bifoncteur  $\morph{\DMB{S}_{W\leqslant 0}\times\DMB{S}_{W\leqslant 0}}{\DMB{S}_{W\leqslant 0}}$. Si $N\in \DMB{S}_{W\leqslant n}$ et $N'\in \DMB{S}_{W\leqslant n'}$ alors $N[-n], N'[-n']\in \DMBn{S}$ ce qui implique par le raisonnement précédent $N\otimes_S N'[-n-n']=N[-n]\otimes_SN'[-n]\in \DMBn{S}$ soit encore $N\otimes_S N'\in \DMB{S}_{W\leqslant n+n'}$.
\end{description}
\end{proof}

\begin{corollaire} 
Soient $f:\morph{X}{S}$ un morphisme de schémas, $n\in \Z$ et $P\in \DMBp{S}$. Le foncteur 
$\HHom(\bullet,f^!P) : \morph{\DMB{S}^\opp}{\DMB{S}}$
induit un foncteur $\morph{\DMB{S}^{\opp}_{W\leqslant n}}{\DMB{S}_{W\geqslant -n}}$. 
Si de plus $P$ est constructible alors il induit également un foncteur $\morph{\DMBc{S}^{\opp}_{w\leqslant n}}{\DMBc{S}_{w\geqslant -n}}$.

C'est en particulier le cas pour le foncteur de dualité local (\cf \cite[\S 14.3.30]{CD}).
\end{corollaire}

\begin{proof}
C'est un cas particulier de $(iii')$, $(iii'_c)$ du théorème précédent en appliquant $(i')$, $(i'_c)$.
\end{proof}

\begin{proposition}
Supposons que $S$ soit régulier. Soit $$\DMB{S}\supset\L_S:=\left\{f_!\Un_X(x)[2x]\big|x\in \Z,\;f:\morph{X}{S} \textrm{ lisse et propre}\right\}.$$
Notons $\DMBcL{S}:=\CAep{\L_S}$ la catégorie des motifs lisses de Levine (\cf \cite{levine}). Alors il existe $\Lev/\DMBcL{S}$ une structure de poids telle que $\Lev=\restr{w}{\DMBcL{S}}$.
\end{proposition}

\begin{proof}
\cf \ref{Pour.Levine}.
\end{proof}

\section*{Remerciements.}
\addcontentsline{toc}{section}{Remerciements.}
Mes remerciements les plus profonds vont à Jörg Wildeshaus pour m'avoir introduit à l'élégante théorie des motifs. Je le remercie également pour m'avoir suggéré un énoncé simple du lemme de Chow motivique (\ref{Mot.Chow.lem}). Je remercie Denis-Charles Cisinski pour m'en avoir indiqué la preuve. Je remercie également Frédéric Déglise pour toutes les discussions que nous avons entretenues. Je tiens tout particulièrement à remercier Bradley Drew pour m'avoir indiqué une preuve simple de \ref{propo.compatibilite}.$(iii_c)$ ainsi que pour sa patiente écoute et sa relecture scrupuleuse.

\addcontentsline{toc}{section}{Bibliographie}
\bibliographystyle{alpha}
\bibliography{Bibliographie}

\end{document}